\documentclass[a4paper,reqno]{amsart}

\usepackage{geometry}
\title{The Zariski Cotangent Space at the Origin of the Pencil Diffeological Space}
\author{Masaki Taho}

\makeatletter
\@namedef{subjclassname@2020}{\textup{2020} Mathematics Subject Classification}
\makeatother


\address{GRADUATE SCHOOL OF MATHEMATICAL SCIENCES, THE
UNIVERSITY OF TOKYO, 3-8-1 KOMABA, MEGURO-KU, TOKYO, 153-8914,
JAPAN}
\email{taho@ms.u-tokyo.ac.jp}

\usepackage{amsmath}
\usepackage{amssymb}
\usepackage{amsfonts}
\usepackage{amsthm}
\usepackage{mathtools}
\usepackage{mathrsfs}
\usepackage{bm}
\usepackage{xcolor}
\usepackage[backref,pdfusetitle]{hyperref}
\hypersetup{
    pdfauthor={Masaki Taho},
    colorlinks=true,
    linkcolor=[rgb]{0.3,0.55,0.25},
    citecolor=[rgb]{0.164,0.39,0.77},
    urlcolor=[rgb]{0,0,1}
}
\usepackage[nameinlink,capitalize,noabbrev]{cleveref}
\usepackage{thmtools}
\usepackage{thm-restate}

\crefname{thm}{theorem}{theorems}
\Crefname{thm}{Theorem}{Theorems}
\crefname{prop}{proposition}{propositions}
\Crefname{prop}{Proposition}{Propositions}
\crefname{lemma}{lemma}{lemmas}
\Crefname{lemma}{Lemma}{Lemmas}
\crefname{cor}{corollary}{corollaries}
\Crefname{cor}{Corollary}{Corollaries}
\crefname{defi}{definition}{definitions}
\Crefname{defi}{Definition}{Definitions}
\crefname{ex}{example}{examples}
\Crefname{ex}{Example}{Examples}
\crefname{remark}{remark}{remarks}
\Crefname{remark}{Remark}{Remarks}
\crefname{equation}{equation}{equations}
\Crefname{equation}{Equation}{Equations}
\crefname{section}{section}{sections}
\Crefname{section}{Section}{Sections}

\newtheorem{thm}{Theorem}[section]
\newtheorem{prop}[thm]{Proposition}
\newtheorem{lemma}[thm]{Lemma}
\newtheorem{cor}[thm]{Corollary}

\newtheorem*{mthm*}{Main Theorem}
\newtheorem*{thm*}{Theorem}
\newtheorem*{prop*}{Proposition}
\theoremstyle{definition}
\newtheorem{defi}[thm]{Definition}
\newtheorem{ex}[thm]{Example}
\newtheorem{remark}[thm]{Remark}

\newcommand{\R}{\mathbb{R}}
\newcommand{\N}{\mathbb{N}}
\newcommand{\Z}{\mathbb{Z}}
\newcommand{\RP}{\mathbb{RP}}
\newcommand{\germ}{G}
\newcommand{\plotdom}[1]{U_{#1}}
\newcommand{\colim}{\operatorname*{colim}}
\newcommand{\righttangent}[2]{\hat{T}^R_{#1}(#2)}
\newcommand{\external}[2]{\hat{T}_{#1}(#2)}
\newcommand{\smooth}{C^\infty}
\newcommand{\supp}{\operatorname{supp}}
\newcommand{\id}{\operatorname{id}}
\newcommand{\Odd}{\operatorname{Odd}}

\newcommand{\Hom}{\operatorname{Hom}}
\newcommand{\ev}{\operatorname{ev}}

\begin{document}

\begin{abstract}
The pencil diffeological space is a plane whose smooth structure at the origin is detected only along lines through it.
We show that the Zariski cotangent space at the origin is entirely determined by the derivatives along these lines.
Moreover, the derivatives along different lines form precisely a $\sigma$-continuous family, a weak form of continuity.
We also compute the external and right tangent spaces at the origin from this description.
This suggests that the local smooth structure at a singular point of a diffeological space can give rise to unexpected topological notions.
\end{abstract}

\maketitle

\section{Introduction}

Diffeologies were introduced by Souriau in his paper
\emph{Groupes diff\'erentiels} \cite{Souriau} and subsequently expanded upon
in the preprint by Donato and Iglesias, \emph{Exemple de groupes
diff\'erentiels: flots irrationnels sur le tore} \cite{Donato-Iglesias}.
For a general reference, see Iglesias-Zemmour \cite{IZ}.
For a based diffeological space \((X,x)\), the Zariski cotangent space \(I_x(X)/I_x^2(X)\) records the first-order structure at \(x\), where \(I_x(X)\subset\germ(X,x)\) denotes the ideal of smooth germs vanishing at \(x\).

The main object of this paper is the \emph{pencil diffeological space} \(P\), the plane \(\R^2\) whose smooth structure at the origin is detected only along lines through it.
Any smooth map from a Euclidean space into \(P\) whose image passes through \(0\) locally lies in a single line through \(0\).
Consequently, each smooth germ in \(I_0(P)\) has an ordinary derivative along every line \(L\in\RP^1\) through \(0\).
The basic question is what compatibility these linewise derivatives satisfy.

Write \(E_0(P)=I_0(P)/I_0(P)^2\). For \(f\in I_0(P)\) and
\(L\in\RP^1\), let
\[
  \tilde J(f)_L=d_0(f|_L)\in L^*
\]
denote the linewise derivative of \(f\) along \(L\).
The first main result is the following.

\begin{restatable*}{thm}{mainjtheorem}\label{thm:main-J}
The map \(\tilde J\colon I_0(P)\to\prod_{L\in\RP^1}L^*\) is a smooth linear map.
Its kernel satisfies
\[
  \ker \tilde J=I_0(P)^2.
\]
Hence it descends to a smooth injective linear map
\[
  J\colon E_0(P)\hookrightarrow
  \prod_{L\in\RP^1}L^*,
  \qquad
  [F]\mapsto
  \bigl(d_0(F|_L)\bigr)_{L\in\RP^1}.
\]
\end{restatable*}
Thus the collection of linewise derivatives detects the entire class of a germ in \(E_0(P)\).
A germ lies in \(I_0(P)^2\) exactly when all of its linewise derivatives vanish.

At the level of underlying vector spaces, the target \(\prod_{L\in\RP^1}L^*\) can be viewed as the vector space \(\Odd(S^1)\) of odd functions on \(S^1\).
It remains to identify \(J(E_0(P))\) as a subspace of \(\Odd(S^1)\).
In the ordinary plane, the corresponding function on \(S^1\) is the restriction of a single linear functional on \(\R^2\).
For the countable bouquet of lines studied in our previous work \cite{taho-nonsmooth}, the derivatives along the different branches can be chosen independently.
With respect to linewise first-order data, the pencil diffeology lies between these two cases.
Its linewise derivatives need not come from a single linear functional, but they cannot vary independently.
The possible linewise derivative functions are exactly the odd functions on \(S^1\) satisfying a weak form of continuity called \(\sigma\)-continuity.
We denote this space by \(\Odd_\sigma(S^1)\).

\begin{restatable*}{thm}{mainimagetheorem}\label{thm:main-image}
Under this vector-space identification, the following equality holds.
\[
  J(E_0(P))=\Odd_\sigma(S^1).
\]
\end{restatable*}
The new phenomenon is angular regularity.
The Zariski cotangent space records a directionwise first-order invariant whose possible angular behavior is governed by \(\sigma\)-continuity rather than ordinary continuity.
This \(\sigma\)-continuity reflects the standard smooth structure away from the origin, which allows smooth parametrizations to move between different lines.

This description also lets us compute, at the origin of \(P\), the external tangent space of Christensen--Wu \cite[Definition~3.10]{CW}.
For this purpose, we equip \(\Odd_\sigma(S^1)\) with the diffeology transported from \(E_0(P)\) through the identification above.
Here and below, \(\Hom_{\R}^{\infty}(V,W)\) denotes the space of smooth \(\R\)-linear maps \(V\to W\).

\begin{restatable*}{thm}{tangentcomputationtheorem}\label{thm:tangent-computation}
For \(v\in S^1\), denote evaluation at \(v\) on \(\Odd_\sigma(S^1)\) by \(\ev_v\).
Under the identification \(E_0(P)\cong\Odd_\sigma(S^1)\), we have
\[
  \external{0}{P}
  \cong
  \Hom_{\R}^{\infty}\bigl(\Odd_\sigma(S^1),\R\bigr)
  =\operatorname{span}\{\ev_v\mid v\in S^1\}
  \cong\bigoplus_{L\in\RP^1}L.
\]
\end{restatable*}
In particular, \Cref{cor:tangent-separation} gives
\(\external{0}{P}\subsetneq\righttangent{0}{P}\).
In the ordinary plane, the linewise derivatives come from a single linear functional, whereas in the countable bouquet they can be chosen independently.
For \(P\), they are related exactly by \(\sigma\)-continuity.
This \(\sigma\)-continuity links the different lines but still allows algebraic linear functionals on \(E_0(P)\) that are not smooth.

\section{Diffeological Background}

We collect the diffeological background used throughout the paper.
The goal is to make the later arguments readable without assuming prior familiarity with the subject.
General references are \cite{Souriau} and \cite{IZ}.

\begin{defi}[{\cite[1.5]{IZ}}]\label{def:diffeology}
  Let \(X\) be a set.
  A \emph{parametrization} of \(X\) is a map \(U\to X\), where \(U\) is an open subset of \(\R^n\) for some \(n\ge0\).
  A \emph{diffeology} on \(X\) is a set of parametrizations \(\mathscr{D}_X\), whose elements are called plots, satisfying the following three axioms.
  \begin{description}
    \item[Covering] Every constant parametrization \(U\to X\) is a plot.
    \item[Locality] Let \(p\colon U\to X\) be a parametrization. If there is an open covering \(\{U_\alpha\}\) of \(U\) such that \(p|_{U_\alpha}\in\mathscr{D}_X\) for all \(\alpha\), then \(p\) itself is a plot.
    \item[Smooth compatibility] For every plot \(p\colon U\to X\), every open set \(V\subset\R^m\), and every ordinary smooth map \(f\colon V\to U\), the composite \(p\circ f\colon V\to X\) is also a plot.
  \end{description}
  A diffeological space is a set equipped with a diffeology.
  We usually write \(X\) for the diffeological space \((X,\mathscr{D}_X)\).
\end{defi}

\begin{defi}[{\cite[1.14]{IZ}}]
A map \(f\colon X\to Y\) between diffeological spaces is \emph{smooth} if for every plot \(p\colon \plotdom{p}\to X\), the composition \(f\circ p\) is a plot of \(Y\).
\end{defi}

\begin{defi}[{\cite[2.8]{IZ}}]\label{def:D-topology}
  Let \(X\) be a diffeological space.
  The \emph{\(D\)-topology} on \(X\) is the finest topology that makes all plots \(p\colon \plotdom{p}\to X\) continuous.
\end{defi}

We use the following standard diffeologies.
The product diffeology on \(\prod_a X_a\) is defined componentwise.
The subset diffeology on \(Y\subset X\) is defined by the plots \(p\colon \plotdom{p}\to Y\) whose composite with \(Y\hookrightarrow X\) is a plot of \(X\).
The quotient diffeology on a quotient set \(X/{\sim}\) is the final diffeology for the quotient map \(X\to X/{\sim}\).

If \(X\) and \(Y\) are diffeological spaces, the functional diffeology on \(\smooth(X,Y)\) is defined by declaring \(q\colon \plotdom{q}\to\smooth(X,Y)\) to be a plot when the adjoint map
\[
  \plotdom{q}\times X\longrightarrow Y,\qquad (u,x)\mapsto q(u)(x),
\]
is smooth.

\subsection{Germ algebras, Zariski cotangent spaces, and tangent spaces}

\begin{defi}[Germ algebra and Zariski cotangent space]\label{def:germ-algebra}
  Let \((X,x)\) be a based diffeological space.
  Following \cite[paragraph before Definition~3.10]{CW}, the set of germs of smooth functions at~\(x\) is defined as the colimit
  \[
    \germ(X,x)=\colim_{x\in B\subset X}\smooth(B,\R),
  \]
  where the colimit ranges over all \(D\)-open neighborhoods \(B\) of \(x\).
  The space \(\germ(X,x)\) carries the quotient diffeology induced from the coproduct of the functional diffeologies on the spaces \(\smooth(B,\R)\), and becomes a diffeological \(\R\)-algebra under the pointwise operations.

  We write \(I_x(X)=\ker(\ev_x\colon\germ(X,x)\to\R)\), where \(\ev_x(f)=f(x)\), and equip \(I_x(X)\) with the subdiffeology from \(\germ(X,x)\).
  Here \(I_x^2(X)\) denotes the ideal consisting of finite sums of products \(fg\) with \(f,g\in I_x(X)\).
  We write \(E_x(X)=I_x(X)/I_x^2(X)\) for the Zariski cotangent space at \(x\), equipped with the quotient diffeology.
\end{defi}

\begin{remark}
The adjective \emph{Zariski} distinguishes \(E_x(X)\) from the cotangent
space \(\Lambda_x^1(X)\) obtained from diffeological differential \(1\)-forms
\cite[6.48]{IZ}. For a smooth manifold \(M\), both spaces canonically agree
with the ordinary cotangent space at \(x\).
\end{remark}

We next recall the external tangent space and the right tangent space.
For the terminology of right tangent spaces we follow our previous work
\cite{taho-variants}, and for external tangent spaces we follow \cite{CW}.

\begin{defi}\label{def:external-right}
  Let \((X,x)\) be a based diffeological space and let \(\germ(X,x)\) be the diffeological \(\R\)-algebra of germs of smooth functions at \(x\) (\Cref{def:germ-algebra}).
  A map \(D\colon\germ(X,x)\to\R\) is said to be a \emph{right tangent vector} if it is linear and satisfies the Leibniz rule
  \[
    D(gh)=g(x)D(h)+D(g)h(x)\quad \text{for all }g,h\in\germ(X,x).
  \]
  The set of all right tangent vectors, denoted by \(\righttangent{x}{X}\), forms a vector space called the \emph{right tangent space of \(X\) at \(x\)} \cite{taho-variants}.

  An element \(D\in\righttangent{x}{X}\) is further called an \emph{external tangent vector} if \(D\) is smooth as a map of diffeological spaces.
  The subspace of all external tangent vectors, denoted by \(\external{x}{X}\), is called the \emph{external tangent space of \(X\) at \(x\)} \cite[Definition~3.10]{CW}.
\end{defi}

\begin{prop}[{\cite[Proposition~3.11]{CW} and \cite[Proposition~2.14]{taho-nonsmooth}}]\label{prop:external-right-Ix}
  Let \((X,x)\) be a based diffeological space.
  With the notation above, there are natural vector-space isomorphisms
  \[
    \external{x}{X}\cong
    \Hom_{\R}^{\infty}(E_x(X),\R)
    =
    \{F\colon E_x(X)\to\R\mid F\text{ is smooth and \(\R\)-linear}\},
  \]
  and
  \[
    \righttangent{x}{X}\cong
    \Hom_{\R}(E_x(X),\R)
    =
    \{F\colon E_x(X)\to\R\mid F\text{ is \(\R\)-linear}\}.
  \]
\end{prop}

\section{The Pencil Diffeology and Linewise Derivatives}

\begin{defi}[Pencil diffeology]\label{def:pencil}
On the set \(\R^2\), we say that a parametrization \(p\colon U\to\R^2\) is a plot if, for every \(u_0\in U\), there is an open neighborhood \(V\subset U\) of \(u_0\) such that one of the following conditions holds.
\begin{enumerate}
  \item \(p(u_0)\neq0\), and \(p|_V\colon V\to\R^2\) is ordinary smooth.
  \item \(p(u_0)=0\), and there is a line \(L\subset\R^2\) through the origin such that \(p(V)\subset L\) and \(p|_V\colon V\to L\) is ordinary smooth.
\end{enumerate}
These plots satisfy the three axioms of a diffeology.
We call this diffeology the \emph{pencil diffeology} and write \(P\) for \(\R^2\) equipped with it.
\end{defi}

Every rotation of \(\R^2\) is a diffeomorphism of \(P\), since it maps
each line through the origin to another such line. Nevertheless, the action
\[
  SO(2)\times P\longrightarrow P,\qquad (R,x)\mapsto Rx,
\]
is not smooth. Indeed, if \(R_s\) denotes rotation through angle \(s\),
smoothness of the action would imply that
\[
  \R^2\longrightarrow P,\qquad
  (s,t)\mapsto R_s(te_1)=t(\cos s,\sin s)
\]
is a plot. Its image near \((0,0)\), however, is not contained in any single
line through the origin.

\begin{lemma}\label{lem:D-open}
For a subset \(O\subset P\), the following conditions are equivalent.
\begin{enumerate}
  \item \(O\) is \(D\)-open.
  \item \(O^\times=O\cap(\R^2\setminus\{0\})\) is open in the usual topology of \(\R^2\setminus\{0\}\), and \(O\cap L\) is open in the usual topology of \(L\) for every line \(L\subset\R^2\) through the origin.
\end{enumerate}
\end{lemma}

\begin{proof}
If \(O\) is \(D\)-open, the two ordinary openness conditions follow by pulling back \(O\) along the plots \(\R^2\setminus\{0\}\to P\) and \(L\to P\) for each line \(L\subset\R^2\) through the origin.

Conversely, assume the two ordinary openness conditions.
Let \(p\colon \plotdom{p}\to P\) be a plot.
We prove that \(p^{-1}(O)\) is open.
Fix \(u_0\in p^{-1}(O)\).
If \(p(u_0)\neq0\), then there is an open neighborhood \(W\subset\plotdom{p}\) of \(u_0\) on which \(p\) is an ordinary smooth map into \(\R^2\setminus\{0\}\).
Since \(p(u_0)\in O^\times\) and \(O^\times\) is ordinary open, \(W\cap p^{-1}(O^\times)\) is an open neighborhood of \(u_0\) contained in \(p^{-1}(O)\).
Suppose instead that \(p(u_0)=0\).
Then we have \(0\in O\).
By the definition of the pencil diffeology, we can choose an open neighborhood \(W\subset\plotdom{p}\) of \(u_0\) and a line \(L\subset\R^2\) through the origin such that \(p(W)\subset L\) and \(p|_W\colon W\to L\) is ordinary smooth.
Since \(O\cap L\) is ordinary open in \(L\), the set \(W\cap p^{-1}(O\cap L)\) is an open neighborhood of \(u_0\) contained in \(p^{-1}(O)\).
Thus every point of \(p^{-1}(O)\) has an open neighborhood contained in \(p^{-1}(O)\), so \(p^{-1}(O)\) is open.
Since this holds for every plot \(p\), the set \(O\) is \(D\)-open.
\end{proof}

\begin{ex}\label{ex:D-open-not-usual-open}
The \(D\)-topology of \(P\) is strictly finer than the ordinary topology of \(\R^2\).
Indeed, every ordinary open subset of \(\R^2\) is \(D\)-open, since every plot of \(P\) is locally ordinary continuous as a map to \(\R^2\).
Let
\[
  C=\{(1/n,1/n^2)\in \R^2\mid n\ge1\}
\]
and set \(O=\R^2\setminus C\).
Then \(O^\times=O\cap(\R^2\setminus\{0\})\) is open in \(\R^2\setminus\{0\}\), because \(C\) has no accumulation point away from \(0\).
Moreover, each line through the origin meets \(C\) in at most one point, since the points of \(C\) have pairwise distinct slopes \(1/n\).
Hence \(O\cap L\) is open in \(L\) for every line \(L\) through the origin.
Thus \(O\) is \(D\)-open by \Cref{lem:D-open}.
However, \(O\) is not ordinary open in \(\R^2\), since \(0\in O\) and every ordinary neighborhood of \(0\) meets \(C\).
\end{ex}

\begin{lemma}\label{lem:pencil-smooth-test}
Let \(O\subset P\) be \(D\)-open, and let \(F\colon O\to\R\) be a function.
Then \(F\) is smooth if and only if the following two conditions hold.
\begin{enumerate}
  \item \(F|_{O^\times}\) is ordinary smooth on \(O^\times=O\cap(\R^2\setminus\{0\})\).
  \item \(F|_{O\cap L}\) is ordinary smooth for every line \(L\subset\R^2\) through the origin.
\end{enumerate}
\end{lemma}

\begin{proof}
Assume first that \(F\) is smooth.
On \(O^\times\), the inclusion \(O^\times\to P\) is locally a plot of \(P\), so \(F|_{O^\times}\) is ordinary smooth.
For each line \(L\) through the origin, the inclusion \(O\cap L\to P\) is locally a plot of \(P\), so \(F|_{O\cap L}\) is ordinary smooth.

Now assume that the two stated conditions hold.
Let \(p\colon \plotdom{p}\to O\) be a plot.
We prove that \(F\circ p\) is ordinary smooth locally on \(\plotdom{p}\).
Fix \(u_0\in\plotdom{p}\).
If \(p(u_0)\neq0\), then after shrinking around \(u_0\), the map \(p\) is an ordinary smooth map into \(O^\times\).
Thus \(F\circ p\) is ordinary smooth near \(u_0\).
If \(p(u_0)=0\), then after shrinking around \(u_0\), the image of \(p\) lies in a line \(L\) through the origin and \(p\) is ordinary smooth as a map into \(L\).
Since \(F|_{O\cap L}\) is ordinary smooth, again \(F\circ p\) is ordinary smooth near \(u_0\).
Hence \(F\circ p\) is ordinary smooth for every plot \(p\), so \(F\) is smooth.
\end{proof}

\begin{ex}
The following example shows that a smooth function on \(P\) need not be ordinary smooth on \(\R^2\).
The function
\[
  f(0,0)=0,\qquad
  f(x,y)=\frac{x^3}{x^2+y^2}\quad ((x,y)\neq(0,0))
\]
is smooth as a map \(P\to\R\).
It is ordinary smooth away from \((0,0)\), and for every \((a,b)\neq(0,0)\) its restriction to the line \(t\mapsto t(a,b)\) is
\[
  f(ta,tb)=t\frac{a^3}{a^2+b^2},
\]
which is smooth in \(t\).
Thus \(f\) is smooth by \Cref{lem:pencil-smooth-test}.
It is not ordinary smooth at \((0,0)\), since its directional derivatives in the directions \((1,0)\), \((0,1)\), and \((1,1)\) are \(1\), \(0\), and \(1/2\), respectively, which is incompatible with linearity.
\end{ex}

\subsection{Derivatives along lines through the origin}

We denote \(\germ(P,0)\) by \(A\), \(I_0(P)\) by \(I\), and \(E_0(P)=I/I^2\) by \(E\).
For \(L\in\RP^1\), regarded as a line through the origin in \(\R^2\), write \(L^*\) for its ordinary dual.
For \(f\in I\), define \(\tilde J(f)\in\prod_{L\in\RP^1}L^*\) by setting its \(L\)-component to be the ordinary derivative \(d_0(f|_L)\in L^*\).

\begin{prop}\label{prop:J-smooth}
The map
\[
  \tilde J\colon I\longrightarrow \prod_{L\in\RP^1}L^*
\]
is well defined, linear, and smooth.
\end{prop}

\begin{proof}
The definition is independent of the representative of the germ, because two representatives agree on some \(D\)-open neighborhood of \(0\), and by \Cref{lem:D-open} this neighborhood restricts to an ordinary neighborhood of \(0\) in every line \(L\).
Linearity is immediate.

It remains to check smoothness.
Since the target has the product diffeology, it is enough to check every component.
Fix a line \(L=\R w\), with \(w\neq0\), and evaluate \(L^*\) at \(w\).
Evaluation at \(w\) is a linear diffeomorphism \(L^*\cong\R\).
Let \(q\colon \plotdom{q}\to I\) be a plot, and fix \(u_*\in\plotdom{q}\).
Since \(I\) has the subdiffeology from \(A=\germ(P,0)\), the map \(q\) is also a plot of \(A\).
By the quotient diffeology defining the germ algebra, after shrinking around \(u_*\) there are an open neighborhood \(U_0\subset\plotdom{q}\), a \(D\)-open neighborhood \(O\) of \(0\), and a smooth family \[F\colon U_0\times O\to\R\] such that \(q(u)\) is represented by the germ of \(F(u,-)\) at \(0\).
Since \(q(u)\in I\), we have \(F(u,0)=0\) for all \(u\in U_0\).
Let \(I_w=\{t\in\R\mid tw\in O\}\).
By \Cref{lem:D-open}, \(I_w\) is an open neighborhood of \(0\) in \(\R\).
The map \((u,t)\mapsto(u,tw)\) from \(U_0\times I_w\) to \(U_0\times O\) is smooth.
Hence \((u,t)\mapsto F(u,tw)\) is ordinary smooth on \(U_0\times I_w\).
Therefore
\[
  u\mapsto
  \left.\frac{d}{dt}\right|_{t=0}F(u,tw)
\]
is smooth on \(U_0\).
This is the \(w\)-evaluation of the \(L\)-component of \(\tilde J\circ q\).
\end{proof}

\begin{lemma}[One-dimensional Hadamard lemma]\label{lem:hadamard}
Let \(U\subset\R\) be an open neighborhood of \(0\), and let \(\varphi\in\smooth(U)\) satisfy \(\varphi(0)=0\).
After shrinking \(U\), there exists \(\psi\in\smooth(U)\) such that
\[
  \varphi(t)=t\psi(t).
\]
Moreover, we have \(\psi(0)=\varphi'(0)\).
In particular, if \(\varphi'(0)=0\), then the function
\[
  H(t)=
  \begin{cases}
    \varphi(t)/t, & t\neq0,\\
    0, & t=0
  \end{cases}
\]
is smooth near \(0\).
\end{lemma}

\begin{proof}
For \(U\) sufficiently small, use the function \(\psi(t)=\int_0^1\varphi'(st)\,ds\).
Then we have \(\varphi(t)=t\psi(t)\) and \(\psi(0)=\varphi'(0)\).
The final statement is the same factorization in the case \(\psi(0)=0\).
\end{proof}

\mainjtheorem

\begin{proof}
By \Cref{prop:J-smooth}, the map \(\tilde J\) is well defined, linear, and smooth.
It remains to compute its kernel.

If \(f\in I^2\), then we can write
\[
  f=\sum_{j=1}^N g_jh_j
\]
with \(g_j,h_j\in I\).
Choose representatives of \(f,g_j,h_j\) on a common \(D\)-open neighborhood \(O\) of \(0\), so that the displayed equality holds on \(O\).
For any line \(L\subset\R^2\) through the origin, \Cref{lem:D-open} says that \(O\cap L\) is an ordinary open neighborhood of \(0\) in \(L\).
Thus the restrictions of \(g_j\) and \(h_j\) to \(O\cap L\) define ordinary smooth germs on \(L\), and these germs vanish at \(0\).
It follows that \(d_0((g_jh_j)|_L)=0\) for all \(j\), and hence that \(\tilde J(f)=0\).
This proves \(I^2\subset\ker\tilde J\).

Conversely, let \(f\in I\) and suppose that \(\tilde J(f)=0\).
Choose a representative \(F\colon O\to\R\) on a \(D\)-open neighborhood \(O\subset P\) of \(0\), with \(F(0)=0\).
Define functions \(A_1,A_2\colon O\to\R\) by
\[
  A_i(0)=0,\qquad
  A_i(x)=\frac{x_iF(x)}{x_1^2+x_2^2}\quad (x\neq0).
\]
Then the following factorization holds:
\[
  F=x_1A_1+x_2A_2.
\]
By \Cref{lem:pencil-smooth-test}, the coordinate functions \(x_1,x_2\colon P\to\R\) are smooth; since they vanish at \(0\), they lie in \(I\).
It is therefore enough to prove that \(A_1,A_2\) are smooth germs vanishing at \(0\).

By \Cref{lem:pencil-smooth-test}, it remains to check the restrictions of \(A_i\) to the punctured set and to each line through the origin.
On \(O^\times\), the functions \(A_i\) are ordinary smooth.
Fix \(w\in\R^2\setminus\{0\}\).
By \Cref{lem:D-open}, the intersection \(O\cap\R w\) is an ordinary neighborhood of \(0\) in \(\R w\).
Since the inclusion \(\R w\to P\) is a plot and \(F\) is smooth on \(O\), the function
\[
  \varphi_w(t)=F(tw)
\]
is ordinary smooth for \(t\) near \(0\).
It satisfies \(\varphi_w(0)=0\), and the assumption \(\tilde J(f)=0\) gives \(\varphi_w'(0)=0\).
By \Cref{lem:hadamard},
\[
  H_w(t)=
  \begin{cases}
    F(tw)/t, & t\neq0,\\
    0, & t=0
  \end{cases}
\]
is smooth near \(0\).
For \(t\) near \(0\), we have
\[
  A_i(tw)=\frac{w_i}{|w|^2}H_w(t).
\]
Therefore \(A_i|_{O\cap\R w}\) is ordinary smooth near \(0\), and away from \(0\) this restriction is ordinary smooth by the formula defining \(A_i\).
It follows that \(A_i\) is smooth, and hence that \(f\in I^2\).
This proves \(\ker\tilde J\subset I^2\); combined with the first inclusion, we get \(\ker\tilde J=I^2\).

The asserted map \(J\) is the map induced by \(\tilde J\) on the quotient \(E=I/I^2\).
It is smooth and linear because \(\tilde J\) is smooth and linear, and it is injective by the kernel computation.
\end{proof}

\section{The Image of the Derivative Map}

We now identify the underlying vector space \(\prod_{L\in\RP^1}L^*\) with a space of odd functions on \(S^1\).
Write \(\Odd(S^1)\) for the vector space of odd functions \(S^1\to\R\).
A family \(s=(s_L)_{L\in\RP^1}\) determines
\[
  \hat{s}\colon S^1\to\R,\qquad
  \hat{s}(v)=s_{\R v}(v).
\]
This function is odd, because replacing \(v\) by \(-v\) keeps the same line and changes the value of the corresponding linear functional by a sign.
Conversely, an odd function \(h\colon S^1\to\R\) determines a family of covectors \(s^h=(s^h_L)_{L\in\RP^1}\) by
\[
  s^h_L(tv)=t h(v)
  \qquad (v\in S^1\cap L,\ t\in\R).
\]
The oddness of \(h\) makes this independent of the choice of the unit vector \(v\in S^1\cap L\).
These two constructions are inverse to each other, and we use this vector-space identification throughout the section.
With this identification, the odd function associated to \(J([F])\) is the function
\[
  v\mapsto \left.\frac{d}{dt}\right|_{t=0}F(tv).
\]

\begin{defi}[\(\sigma\)-continuous functions]\label{def:sigma-continuous}
In this paper, a function \(h\colon S^1\to\R\) is called \emph{\(\sigma\)-continuous} if there is a sequence \((C_n)_{n\in\N}\) of closed subsets of \(S^1\) such that
\[
  S^1=\bigcup_{n\in\N}C_n
\]
and \(h|_{C_n}\) is continuous for every \(n\in\N\).
The same closed-cover condition is known as piecewise continuity in the literature around the Jayne--Rogers theorem \cite{Jayne-Rogers}, as \(\sigma\)-continuity with closed witnesses in \cite{Carroy-Miller}, and as \(\bar\sigma\)-continuity in the terminology of \cite{Banakh-sigma-continuous}.
The odd \(\sigma\)-continuous functions form a vector subspace of \(\Odd(S^1)\), denoted by \(\Odd_\sigma(S^1)\).
\end{defi}

\mainimagetheorem
This is an equality of vector subspaces of the odd functions on \(S^1\).
We prove the theorem by separating the two inclusions, stated as \Cref{prop:image-subset,prop:image-superset}.

\subsection{From Germs to \texorpdfstring{\(\boldsymbol{\sigma}\)}{sigma}-Continuous Functions}

\begin{prop}\label{prop:image-subset}
With the identification above, the following inclusion holds.
\[
  J(E)\subset\Odd_\sigma(S^1).
\]
\end{prop}

\begin{proof}
Let \([F]\in E\), where \(F\colon O\to\R\) is a smooth representative on a \(D\)-open neighborhood \(O\) of \(0\), with \(F(0)=0\).
The corresponding odd function \(h\colon S^1\to\R\) is defined by
\[
  h(v)=\left.\frac{d}{dt}\right|_{t=0}F(tv)
  \qquad (v\in S^1).
\]

We define \(\Omega\) and \(H\) by
\[
  \Omega=\{(r,v)\in(0,\infty)\times S^1\mid rv\in O\},
  \qquad
  H\colon\Omega\to\R,\quad H(r,v)=\frac{F(rv)}{r}.
\]
Since \(O^\times\) is ordinary open in \(\R^2\setminus\{0\}\), the set \(\Omega\) is ordinary open in \((0,\infty)\times S^1\).
Since \(F\) is ordinary smooth away from \(0\), the function \(H\) is smooth.
For each fixed \(v\), the function \(r\mapsto H(r,v)\) has limit \(h(v)\) as \(r\to0\) with \(r>0\).

For positive integers \(M,N\), define \(C_{M,N}\subset S^1\) by
\[
  C_{M,N}
  =
  \left\{
  v\in S^1\,\middle|\,
  \begin{array}{l}
    \text{for all }r,s\in(0,1/N)\text{ with }(r,v),(s,v)\in\Omega,\\
    |H(r,v)-H(s,v)|\le M(r+s)
  \end{array}
  \right\}.
\]

We first prove that each \(C_{M,N}\) is closed.
Let \(v_j\in C_{M,N}\) and suppose that \(v_j\to v\) in \(S^1\).
Take \(r,s\in(0,1/N)\) such that \((r,v),(s,v)\in\Omega\).
Since \(\Omega\) is open, there is \(j_0\) such that \((r,v_j),(s,v_j)\in\Omega\) for all \(j\ge j_0\).
For such \(j\), the definition of \(C_{M,N}\) gives
\[
  |H(r,v_j)-H(s,v_j)|\le M(r+s).
\]
Letting \(j\to\infty\) and using the continuity of \(H\), we obtain
\[
  |H(r,v)-H(s,v)|\le M(r+s).
\]
This shows that \(v\in C_{M,N}\).

We next show that the closed sets \(C_{M,N}\) cover \(S^1\).
Fix \(v\in S^1\).
It is enough to find positive integers \(M,N\) such that \(v\in C_{M,N}\).
By \Cref{lem:D-open}, there is \(\epsilon>0\) such that \(tv\in O\) whenever \(|t|<\epsilon\).
Choose \(N\) with \(1/N<\epsilon\).
For \(|r|<\epsilon\), define
\[
  G_v(r)=
  \begin{cases}
    F(rv)/r, & r\neq0,\\
    h(v), & r=0.
  \end{cases}
\]
By the one-dimensional Hadamard lemma, \(G_v\) is smooth.
Choose an integer \(M\) such that \(|G_v'(r)|\le M\) on \([0,1/N]\).
For \(r,s\in(0,1/N)\), the mean value theorem gives
\[
  |H(r,v)-H(s,v)|
  =
  |G_v(r)-G_v(s)|
  \le M|r-s|
  \le M(r+s).
\]
Hence we have \(v\in C_{M,N}\).

It remains to prove that \(h\) is continuous on each \(C_{M,N}\).
Fix \(C=C_{M,N}\), \(v_0\in C\), and \(\eta>0\).
Choose \(r_0>0\) such that \(r_0<1/N\), \((r_0,v_0)\in\Omega\), and \(Mr_0<\eta/3\).
By openness of \(\Omega\) and continuity of \(H\), there is a neighborhood \(U\) of \(v_0\) in \(S^1\) such that \((r_0,v)\in\Omega\) for \(v\in U\), and
\[
  |H(r_0,v)-H(r_0,v_0)|<\eta/3.
\]
For \(v\in U\cap C\), \Cref{lem:D-open} gives \(\delta_v>0\) such that \((s,v)\in\Omega\) whenever \(0<s<\delta_v\).
Applying the defining inequality for \(C\) with \(r=r_0\) gives
\[
  |H(r_0,v)-H(s,v)|\le M(r_0+s)
  \qquad (0<s<\min\{\delta_v,1/N\}).
\]
Letting \(s\to0\) through positive values gives
\[
  |H(r_0,v)-h(v)|\le Mr_0.
\]
The same argument with \(v=v_0\) gives
\[
  |H(r_0,v_0)-h(v_0)|\le Mr_0.
\]
The triangle inequality then gives
\[
  |h(v)-h(v_0)|\le Mr_0+\eta/3+Mr_0<\eta.
\]
So \(h|_C\) is continuous.
\end{proof}

\subsection{From \texorpdfstring{\(\boldsymbol{\sigma}\)}{sigma}-Continuous Functions to Germs}

\begin{prop}\label{prop:image-superset}
With the identification above, the following inclusion holds.
\[
  \Odd_\sigma(S^1)\subset J(E).
\]
\end{prop}

Before proving \Cref{prop:image-superset}, we look at a concrete instance of the construction.
This example realizes a discontinuous \(\sigma\)-continuous function by an explicit pencil-smooth function.

\begin{ex}\label{ex:discontinuous-sigma-continuous-profile}
Let \(e_1=(1,0)\), and define an odd function \(h\colon S^1\to\R\) by
\[
  h(v)=
  \begin{cases}
    1, & v=e_1,\\
    -1, & v=-e_1,\\
    0, & v\neq \pm e_1.
  \end{cases}
\]
This function is not continuous at \(\pm e_1\).
However, it is \(\sigma\)-continuous.
Indeed, it is continuous on the closed set \(\{\pm e_1\}\), and on each of the closed sets
\[
  \{v\in S^1\mid |v-e_1|\ge 1/n,\ |v+e_1|\ge 1/n\}
  \qquad (n\ge1).
\]
These closed sets cover \(S^1\).

We now realize this function explicitly.
Choose an even smooth function \(\rho\colon\R\to\R\) such that
\[
  \rho(0)=1,
  \qquad
  \rho(u)=0\quad(|u|\ge1).
\]
For instance, one may take
\[
  \rho(u)=
  \begin{cases}
    \exp\left(-\dfrac{u^2}{1-u^2}\right), & |u|<1,\\
    0, & |u|\ge1.
  \end{cases}
\]
Define \(F\colon B(0,1)\to\R\) by \(F(0)=0\), and for \((x,y)\neq(0,0)\),
\[
  F(x,y)=x\,\rho\left(\frac{y}{x^2+y^2}\right).
\]
This function is ordinary smooth on the punctured ball.
We compute its restriction to each line through the origin.
For \(w=(a,b)\neq0\) and \(t\neq0\), we have
\[
  F(ta,tb)
  =
  ta\,
  \rho\left(\frac{b}{t(a^2+b^2)}\right).
\]
If \(b=0\), this gives
\[
  F(ta,0)=ta.
\]
If \(b\neq0\), then for sufficiently small \(|t|\) the argument of \(\rho\) has absolute value at least \(1\), so that
\[
  F(ta,tb)=0
\]
near \(t=0\).
Thus every line restriction is smooth near the origin.
By \Cref{lem:pencil-smooth-test}, \(F\) is pencil-smooth near the origin.
For \(v\in S^1\), the linewise derivative function is
\[
  \left.\frac{d}{dt}\right|_{t=0}F(tv)
  =
  h(v).
\]
Such a linewise derivative function cannot arise from an ordinary smooth germ on \(\R^2\), since it would be induced by a single linear functional on \(\R^2\), and hence would be continuous on \(S^1\).
\end{ex}

\begin{proof}[Proof of \Cref{prop:image-superset}]

\begin{lemma}\label{lem:log-partition}
There exists a smooth partition of unity \((\lambda_m)_{m\ge2}\) on \((0,e^{-2})\) with the following properties.
\begin{enumerate}
  \item \(\sum_{m\ge2}\lambda_m(r)=1\) for every \(0<r<e^{-2}\).
  \item \(\supp(\lambda_m)\subset(e^{-(m+1)},e^{-(m-1)})\).
  \item for every integer \(k\ge0\), there is a constant \(R_k>0\) such that
  \[
    \left|\frac{d^k}{dr^k}\lambda_m(r)\right|\le R_k r^{-k}
  \]
  for all \(m\ge2\) and all \(r\in(0,e^{-2})\).
\end{enumerate}
\end{lemma}

\begin{proof}
Choose a nonnegative function \(\psi\in C^\infty_c((-1,1))\) such that \(\psi(t)=1\) for \(|t|\le 1/2\).
Define \(S\colon\R\to\R\) by \(S(t)=\sum_{n\in\Z}\psi(t-n)\).
This sum is locally finite.
Moreover, \(S\) is smooth, \(1\)-periodic, and positive, since every \(t\in\R\) satisfies \(|t-n|\le 1/2\) for some \(n\in\Z\).
Define
\[
  \varphi(t)=\frac{\psi(t)}{S(t)}.
\]
Then we have \(\varphi\in C^\infty_c((-1,1))\) and \(\varphi\ge0\), and the periodicity of \(S\) gives
\[
  \sum_{m\in\Z}\varphi(t-m)=1.
\]
For \(m\ge2\), define
\[
  \mu_m(t)=\varphi(t-m)\qquad (t>2).
\]
If \(t>2\) and \(m\le1\), then we have \(t-m>1\), so that \(\varphi(t-m)=0\).
Therefore the partition identity
\[
  \sum_{m\ge2}\mu_m(t)=1\qquad (t>2)
\]
holds.
Also, we have $\supp(\mu_m)\subset(m-1,m+1)$.
Since the functions \(\mu_m\) are translates of the single function \(\varphi\), for every \(j\ge0\) there is a constant \(B_j\) such that
\[
  |\mu_m^{(j)}(t)|\le B_j
  \qquad (m\ge2,\ t>2).
\]

Now put \(t=-\log r\), and define
\[
  \lambda_m(r)=\mu_m(-\log r)=\varphi(-\log r-m)
  \qquad (0<r<e^{-2}).
\]
Since \(-\log r>2\), the identity \(\sum_{m\ge2}\mu_m(t)=1\) gives \(\sum_{m\ge2}\lambda_m(r)=1\).
Since \(\supp(\varphi)\) is a compact subset of \((-1,1)\), choose \(a<1\) such that \(\supp(\varphi)\subset[-a,a]\).
Then the following implications hold:
\[
  \lambda_m(r)\neq0
  \quad\Longrightarrow\quad
  -\log r\in[m-a,m+a]
  \quad\Longrightarrow\quad
  r\in[e^{-(m+a)},e^{-(m-a)}].
\]
Thus we obtain \(\supp(\lambda_m)\subset(e^{-(m+1)},e^{-(m-1)})\).
This also shows local finiteness.

It remains to prove the derivative estimate.
By induction on \(k\), there are constants \(c_{k,j}\), independent of \(m\), such that
\[
  \frac{d^k}{dr^k}\lambda_m(r)
  =
  r^{-k}\sum_{j=1}^k c_{k,j}\,\mu_m^{(j)}(-\log r)
  \qquad (k\ge1).
\]
The uniform bounds for the derivatives of the \(\mu_m\)'s therefore imply
\[
  \left|\frac{d^k}{dr^k}\lambda_m(r)\right|
  \le
  r^{-k}\sum_{j=1}^k |c_{k,j}|B_j
  \qquad (k\ge1).
\]
The case \(k=0\) follows from the boundedness of \(\varphi\).
Thus the required constants \(R_k\) exist.
\end{proof}

Let \(h\in\Odd_\sigma(S^1)\).
Choose closed sets \(K_n\subset S^1\) such that \(S^1=\bigcup_nK_n\) and \(h|_{K_n}\) is continuous.
Replace them by the symmetric increasing closed sets
\[
  C_m=\bigcup_{n\le m}(K_n\cup(-K_n)).
\]
Then we have \(C_m=-C_m\), \(C_m\subset C_{m+1}\), and \(\bigcup_m C_m=S^1\); moreover, \(h|_{C_m}\) is continuous.

For each \(m\), use the Tietze extension theorem to extend \(h|_{C_m}\) to a continuous function \(b_m\colon S^1\to\R\), and replace \(b_m\) by its odd part.
The resulting continuous odd function still agrees with \(h\) on \(C_m\).
By the standard density of \(C^\infty(S^1)\) in \(C(S^1)\), with the approximation chosen odd, we choose smooth odd functions \(\alpha_m\colon S^1\to\R\) satisfying
\[
  |\alpha_m(v)-h(v)|\le e^{-m^2}\qquad (v\in C_m).
\]

Use the partition of unity \((\lambda_m)_{m\ge2}\) from \Cref{lem:log-partition}.
For \(0<r<e^{-2}\) and \(v\in S^1\), set \(A(r,v)=\sum_{m\ge2}\lambda_m(r)\alpha_m(v)\).
The sum is locally finite, so \(A\) is smooth.
Since all \(\alpha_m\) are odd, the function \(A\) satisfies \(A(r,-v)=-A(r,v)\).

Fix \(v\in S^1\), and set
\[
  B_v(r)=A(r,v)-h(v).
\]
We prove that \(B_v\) is flat at \(0\), i.e., after extending \(B_v\) by \(B_v(0)=0\), the extension is smooth at \(0\) and all of its derivatives there are zero.
Since \((\lambda_m)_{m\ge2}\) is a partition of unity, for \(0<r<e^{-2}\) we have
\[
  B_v(r)=\sum_{m\ge2}\lambda_m(r)\bigl(\alpha_m(v)-h(v)\bigr).
\]
Choose \(m_0\) such that \(v\in C_m\) for all \(m\ge m_0\).

Fix integers \(k,N\ge0\), and put \(\lambda_m^{(0)}=\lambda_m\).
For \(0<r<e^{-2}\), define
\[
  M_k(r)=\{m\ge2\mid \lambda_m^{(k)}(r)\neq0\}.
\]
If \(m\in M_k(r)\), then we have \(r\in\supp(\lambda_m^{(k)})\subset\supp(\lambda_m)\), and the support condition in \Cref{lem:log-partition} gives \(e^{-(m+1)}<r<e^{-(m-1)}\).
Equivalently, we have \(\log(1/r)-1<m<\log(1/r)+1\), so \(M_k(r)\) has at most three elements.

Choose \(m_1\ge \max\{m_0,2\}\) so large that
\[
  3R_k e^{-m^2}\le e^{-(N+k)(m+1)}
  \qquad (m\ge m_1).
\]
This is possible because \(m^2-(N+k)(m+1)\to\infty\).
Let \(r_0=e^{-(m_1+1)}\).
If \(0<r<r_0\) and \(m\in M_k(r)\), then the left inequality above gives \(m>m_1\).
Hence we have \(v\in C_m\) and
\[
  |\alpha_m(v)-h(v)|\le e^{-m^2}.
\]
Moreover, since \(e^{-(m+1)}<r\), the choice of \(m_1\) gives
\[
  3R_k e^{-m^2}\le e^{-(N+k)(m+1)}\le r^{N+k}.
\]

By local finiteness and the derivative estimate in \Cref{lem:log-partition}, for \(0<r<r_0\) we have
\[
  \left|\frac{d^k}{dr^k}B_v(r)\right|
  \le
  r^{-k}\sum_{m\in M_k(r)} R_k e^{-m^2}
  \le r^{-k}r^{N+k}=r^N.
\]
Hence, the extension \(B_v(0)=0\) is smooth at \(0\) and all of its derivatives at \(0\) vanish.

Define \(F\colon B(0,e^{-2})\to\R\) by
\[
  F(0)=0,\qquad
  F(rv)=rA(r,v)\quad (0<r<e^{-2},\ v\in S^1).
\]
Here we write \(r=|x|\) and \(v=x/|x|\), so this is an ordinary polar expression away from \(0\).
It is ordinary smooth on the punctured ball.
By \Cref{lem:pencil-smooth-test}, it remains to check the restriction of \(F\) to each line through the origin.
Fix \(w\in S^1\).
For \(t\neq0\), we have
\[
  F(tw)=t h(w)+t B_w(|t|),
\]
where the oddness of \(A\) is used for \(t<0\).
Since \(B_w\) is flat at \(0\), the function \(t\mapsto B_w(|t|)\), with value \(0\) at \(t=0\), is smooth.
Therefore \(t\mapsto F(tw)\) is smooth near \(0\).
Thus the restriction of \(F\) to each line through the origin is ordinary smooth, so \(F\) is smooth as a function on \(P\).

Thus \(F\) represents an element of \(I\).
For every \(v\in S^1\), its linewise derivative satisfies
\[
  \left.\frac{d}{dt}\right|_{t=0}F(tv)=h(v).
\]
Hence we have \(h\in J(E)\).
\end{proof}

By \Cref{thm:main-J,thm:main-image}, the map \(J\) identifies \(E=E_0(P)\) with \(\Odd_\sigma(S^1)\) as vector spaces.
In the rest of this section, whenever smoothness of maps to or from \(\Odd_\sigma(S^1)\) is discussed, \(\Odd_\sigma(S^1)\) carries the diffeology transported from \(E_0(P)\) via this identification.
Thus a parametrization \(q\colon U\to\Odd_\sigma(S^1)\) is a plot precisely when it is locally the linewise first derivative of a smooth family of pencil germs.
Equivalently, locally on \(U\), there exist a \(D\)-open neighborhood \(O\subset P\) of \(0\) and a smooth function
\[
  F\colon U_0\times O\to\R
\]
such that \(F(u,0)=0\) and
\[
  q(u)(v)
  =
  \left.\frac{d}{dt}\right|_{t=0}F(u,tv)
  \qquad (u\in U_0,\ v\in S^1).
\]
Indeed, a plot of \(E=I/I^2\) locally lifts to a plot of \(I\), represented
by such a family \(F\) as in the proof of \Cref{prop:J-smooth}; conversely,
every such family defines a plot of \(E\).
We use this identification without changing notation.

\begin{thm}\label{thm:smooth-profile-approximation}
For every \(h\in\Odd_\sigma(S^1)\), there exist \(\epsilon>0\) and a plot
\[
  c\colon(-\epsilon,\epsilon)\to\Odd_\sigma(S^1)
\]
such that \(c(0)=h\) and \(c(s)\in C^\infty_{\mathrm{odd}}(S^1)\) for every \(s\neq0\).
Consequently, \(C^\infty_{\mathrm{odd}}(S^1)\) is dense in \(\Odd_\sigma(S^1)\) with respect to the \(D\)-topology.
\end{thm}

\begin{proof}
Fix \(h\in\Odd_\sigma(S^1)\).
Let
\[
  A\colon(0,e^{-2})\times S^1\to\R
\]
be the smooth function constructed in the proof of \Cref{prop:image-superset}.
For every \(r\), the function \(A(r,-)\) is smooth and odd.
For every \(v\in S^1\),
\[
  B_v(r)=A(r,v)-h(v)
\]
extends to a smooth function at \(r=0\), by setting \(B_v(0)=0\), which is flat at \(0\).

Choose \(\epsilon,\delta>0\) such that \(\sqrt{\epsilon^2+\delta^2}<e^{-2}\), and let \(B(0,\delta)\subset P\) denote the Euclidean open ball with the subset diffeology.
This set is \(D\)-open in \(P\), for instance by \Cref{lem:D-open}.
Define \(\mathcal F\colon(-\epsilon,\epsilon)\times B(0,\delta)\to\R\) by
\[
  \mathcal F(s,0)=0,
  \qquad
  \mathcal F(s,rv)=rA\!\left(\sqrt{s^2+r^2},v\right)
  \quad (0<r<\delta,\ v\in S^1).
\]
This function is ordinary smooth away from the origin of \(P\).
Along a plot through the origin, write the \(P\)-component locally as \(\tau v\), where \(v\in S^1\) is fixed and \(\tau\) is ordinary smooth.
Then we have \(\mathcal F(s,\tau v)=\tau h(v)+\tau B_v\!\left(\sqrt{s^2+\tau^2}\right)\).
The first term is smooth.
For the second term, define
\[
  g_v(u)=
  \begin{cases}
    B_v(\sqrt u), & u\ge0,\\
    0, & u<0.
  \end{cases}
\]
For \(u>0\), each derivative of \(B_v(\sqrt u)\) is a finite sum of terms of the form \(u^{-a/2}B_v^{(\ell)}(\sqrt u)\), and flatness forces every such term to tend to \(0\) at \(u=0\), so \(g_v\) is smooth near \(0\).
Hence \(B_v\!\left(\sqrt{s^2+\tau^2}\right)=g_v(s^2+\tau^2)\) is smooth near \((0,0)\), and multiplying by \(\tau\) preserves smoothness.

Under the above identification, the plot \(c\colon(-\epsilon,\epsilon)\to\Odd_\sigma(S^1)\) represented by \(\mathcal F\) is given by
\[
  c(s)(v)
  =
  \begin{cases}
    A(|s|,v), & s\neq0,\\
    h(v), & s=0.
  \end{cases}
\]
This formula gives \(c(0)=h\) and \(c(s)\in C^\infty_{\mathrm{odd}}(S^1)\) for every \(s\neq0\).
Since every plot is continuous for the \(D\)-topology, we have \(c(s)\to h\) as \(s\to0\), so \(h\) lies in the closure of \(C^\infty_{\mathrm{odd}}(S^1)\), which proves the density statement.
\end{proof}

We now turn to the external tangent space of \(P\) at the origin.
By \Cref{prop:external-right-Ix}, its computation reduces to describing the
smooth dual of \(E_0(P)\cong\Odd_\sigma(S^1)\).

\tangentcomputationtheorem

\begin{proof}
First, each evaluation \(\ev_v\) is smooth.
Indeed, if \(q\colon U\to\Odd_\sigma(S^1)\) is a plot represented locally by \(F\colon U_0\times O\to\R\), then
\[
  u\mapsto q(u)(v)=\left.\frac{d}{dt}\right|_{t=0}F(u,tv)
\]
is smooth.
Thus every finite linear combination of evaluations is smooth.

We prove the converse.
Let \(\Lambda\colon\Odd_\sigma(S^1)\to\R\) be a smooth linear functional.
Let \(C^\infty_{\mathrm{odd}}(S^1)\) carry the subset diffeology from the functional diffeology on \(C^\infty(S^1,\R)\).
The natural inclusion \(C^\infty_{\mathrm{odd}}(S^1)\to\Odd_\sigma(S^1)\) is smooth.
To see this, let \(q\colon U\to C^\infty_{\mathrm{odd}}(S^1)\) be a plot, and write
\(q^\vee\colon U\times S^1\to\R\), given by \(q^\vee(u,v)=q(u)(v)\), for its smooth adjoint.
Define
\[
  \tilde q(u,0)=0,
  \qquad
  \tilde q(u,rv)=r q^\vee(u,v)
  \quad (r>0,\ v\in S^1).
\]
This is smooth as a map \(U\times P\to\R\).
In fact, away from \(0\in P\), this is ordinary smooth.
Along a plot through \(0\), the \(P\)-component has the form \(\tau v\) for a fixed \(v\in S^1\), and the composite becomes \(\tau q^\vee(u,v)\), which is smooth.
The linewise first derivative of this family is \(q\), so \(q\) is also a plot of \(\Odd_\sigma(S^1)\).

We write \(T_{\mathrm{odd}}=\Lambda|_{C^\infty_{\mathrm{odd}}(S^1)}\) for the restriction of \(\Lambda\) to \(C^\infty_{\mathrm{odd}}(S^1)\).
The odd projection
\[
  P_{\mathrm{odd}}\colon C^\infty(S^1)\to C^\infty_{\mathrm{odd}}(S^1),
  \qquad
  (P_{\mathrm{odd}}\varphi)(v)=\frac{\varphi(v)-\varphi(-v)}{2},
\]
is smooth for the functional diffeologies.
Hence \(T(\varphi)=T_{\mathrm{odd}}(P_{\mathrm{odd}}\varphi)\) defines a smooth linear functional on \(C^\infty(S^1)\), and is therefore continuous for the \(D\)-topology.
By \cite[Corollary~4.10]{CSW-D-topology}, the \(D\)-topology on
\(C^\infty(S^1)\) agrees with the weak \(C^\infty\)-topology, which is
known to coincide with the usual Fr\'echet topology when the domain is
compact.
Thus \(T\) is continuous for the Fr\'echet topology, and hence it is a distribution on \(S^1\).

We next show that \(\supp T\) is finite.
By \cite[Proposition~3.2]{taho-nonsmooth}, every smooth linear functional on \(V=\prod_{n\in\N}\R\), equipped with the product diffeology, depends on finitely many coordinates.
Assume now that \(\supp T\) is infinite.
Then \(\supp T\) contains infinitely many antipodal orbits.
Choose a sequence of distinct such orbits converging to an orbit outside the sequence, and choose pairwise disjoint antipodally symmetric open neighborhoods \(W_n=-W_n\), each meeting \(\supp T\).
For each \(n\), choose \(\varphi_n\in C_c^\infty(W_n)\) such that \(T(\varphi_n)\neq0\), and set \(\beta_n=P_{\mathrm{odd}}\varphi_n\).
After rescaling, we may assume that \(T_{\mathrm{odd}}(\beta_n)=1\) for every \(n\in\N\).
The supports of the \(\beta_n\)'s are pairwise disjoint.

Choose an even smooth function \(\rho\colon\R\to[0,1]\) such that \(\rho=1\) on \([-1/2,1/2]\) and \(\supp(\rho)\subset(-1,1)\), and choose positive numbers \(\epsilon_n\to0\).
For \(a=(a_n)\in V\), define \(F_a\colon P\to\R\) by
\[
  F_a(0)=0,
  \qquad
  F_a(rv)
  =
  r\sum_{n=1}^{\infty}
  a_n\rho(r/\epsilon_n)\beta_n(v)
  \quad (r>0,\ v\in S^1).
\]
This is well defined and smooth for the pencil diffeology.
Away from the origin the sum is locally finite, because \(\epsilon_n\to0\).
Along any fixed line through the origin, the disjointness of the supports of the \(\beta_n\)'s leaves at most one nonzero summand.
Define the linear map
\[
  \Psi\colon V\to\Odd_\sigma(S^1),
  \qquad
  \Psi(a)=
  \left[
    v\mapsto
    \left.\frac{d}{dt}\right|_{t=0}F_a(tv)
  \right],
\]
and let \(p\colon U\to V\) be a plot, with \(p(u)=(p_n(u))_n\).
The map \((u,x)\mapsto F_{p(u)}(x)\) from \(U\times P\) to \(\R\) is smooth.
Indeed, away from the origin of \(P\), the defining sum is locally finite, while along a plot through the origin whose \(P\)-component is \(\tau w\), it becomes
\[
  \tau\sum_{n=1}^{\infty}p_n(u)\rho(\tau/\epsilon_n)\beta_n(w),
\]
using the evenness of \(\rho\) and the oddness of the \(\beta_n\)'s.
This sum has at most one nonzero term and is therefore smooth.
Its linewise first derivative is \(\Psi\circ p\), so \(\Psi\circ p\) is a plot and \(\Psi\) is smooth.
In the present identification, we have \(\Psi(a)=\sum_{n=1}^{\infty}a_n\beta_n\) pointwise, and in particular we have \(\Psi(e_n)=\beta_n\).
Thus, for every \(n\in\N\), we have \((\Lambda\circ\Psi)(e_n)=T_{\mathrm{odd}}(\beta_n)=1\).
This contradicts the finite-coordinate property for smooth linear functionals on \(V\).
Hence \(\supp T\) is finite.

Put \(K=\supp T\cup(-\supp T)\), and choose representatives
\(p_1,\ldots,p_N\) for the antipodal pairs in \(K\). For each \(i\), choose
an open arc \(W_i\) containing \(p_i\) such that the \(2N\) sets
\(W_1,-W_1,\ldots,W_N,-W_N\) are pairwise disjoint and
\(K\cap(W_i\cup(-W_i))=\{p_i,-p_i\}\). Choose a local coordinate
\(\theta_i\) on \(W_i\) with \(\theta_i(p_i)=0\).

For \(\psi\in C_c^\infty(W_i)\), let \(S_i\psi\in C^\infty_{\mathrm{odd}}(S^1)\)
be the odd extension, equal to \(\psi\) on \(W_i\), to \(-\psi\circ(-\id)\)
on \(-W_i\), and to \(0\) outside \(W_i\cup(-W_i)\). Set
\(T_i(\psi)=T_{\mathrm{odd}}(S_i\psi)=T(S_i\psi)\). Equip
\(C_c^\infty(W_i)\) with its usual test-function topology. The odd-extension
map \(S_i:C_c^\infty(W_i)\to C^\infty(S^1)\) is continuous into the usual
Fr\'echet topology. Hence \(T_i=T\circ S_i\) is a distribution on \(W_i\).
Its support is contained in \(\{p_i\}\).
Indeed, if \(\psi\) is supported away from \(p_i\), then \(S_i\psi\) is supported
away from \(K\), hence away from \(\supp T\), and therefore \(T_i(\psi)=0\).

The structure theorem for distributions supported at a point
\cite[Section~2.3]{HormanderI} yields an integer \(m_i\ge0\) and coefficients
\(A_{i,j}\in\R\) such that
\[
  T_i(\psi)=\sum_{j=0}^{m_i}A_{i,j}\,\partial_{\theta_i}^j\psi(p_i)
  \qquad(\psi\in C_c^\infty(W_i)).
\]
We claim that all positive-order coefficients vanish. Suppose not, and let
\(m=\max\{j\ge1\mid A_{i,j}\neq0\text{ for some }i\}\). Fix any index \(i\)
such that \(A_{i,m}\neq0\), and write \(p=p_i\), \(W=W_i\), and
\(\theta=\theta_i\).
Let \(\chi\in C_c^\infty((-1,1))\) be equal to \(1\) near \(0\), and put
\(\varphi(x)=x^m\chi(x)\). Choose \(a_0>0\) such that
\(a\supp\varphi\Subset\theta(W)\) for every \(0<a<a_0\). For such \(a\),
define \(\psi_a\in C_c^\infty(W)\) by
\(\psi_a(v)=\varphi(\theta(v)/a)\), and set \(\beta_a=S_i\psi_a\). Since
\(\varphi^{(j)}(0)=0\) for \(0\le j<m\) and \(\varphi^{(m)}(0)=m!\), the
local formula for \(T_i\) gives
\[
  T_{\mathrm{odd}}(\beta_a)=T_i(\psi_a)
  =A_{i,m}m!\,a^{-m}.
\]
Let \(C=A_{i,m}m!\), and note that \(C\neq0\).

Choose \(\epsilon,\delta>0\) such that
\(\sqrt{\epsilon^2+\delta^2}<a_0\).
Let \(B(0,\delta)\subset P\) be the Euclidean open ball with the subset
diffeology. Define \(\mathcal G:(-\epsilon,\epsilon)\times B(0,\delta)\to\R\)
by
\[
  \mathcal G(s,0)=0,\qquad
  \mathcal G(s,x)=|x|\,\beta_{\sqrt{s^2+|x|^2}}\!\left(\frac{x}{|x|}\right)
  \quad(0<|x|<\delta).
\]
This map is smooth. Away from \(0\in P\), this is ordinary smooth. Along a
plot through the origin contained in a fixed line \(\R w\), it has the form
\(t\,\beta_{\sqrt{s^2+t^2}}(w)\). If \(w\neq\pm p\), this expression is
identically zero near \((s,t)=(0,0)\), because \(w\) is outside
\(\supp\beta_a\) for all sufficiently small \(a\). If \(w=\pm p\), it is
\(\pm t\,\varphi(0)=0\), since \(m\ge1\).

Taking the linewise first derivative of the smooth family \(\mathcal G\) gives
a plot \(\gamma:(-\epsilon,\epsilon)\to\Odd_\sigma(S^1)\), defined by
\[
  \gamma(s)(v)=\left.\frac{d}{dt}\right|_{t=0}\mathcal G(s,tv)
  \qquad(v\in S^1).
\]
For \(s\neq0\), the oddness of \(\beta_a\) gives
\(\mathcal G(s,tv)=t\,\beta_{\sqrt{s^2+t^2}}(v)\) for \(t\) near \(0\).
Differentiating at \(t=0\) yields
\(\gamma(s)(v)=\beta_{|s|}(v)\) for every \(v\in S^1\), which proves the
identity \(\gamma(s)=\beta_{|s|}\). The formula for
\(T_{\mathrm{odd}}(\beta_a)\) now gives
\[
  (\Lambda\circ\gamma)(s)=C|s|^{-m}
  \qquad(s\neq0).
\]
The right-hand side is unbounded as \(s\to0\), contradicting the smoothness
of \(\Lambda\circ\gamma\). Hence
\(A_{i,j}=0\) for all \(i\) and
all \(j\ge1\).

Put \(A_i=A_{i,0}\). Choose \(\zeta_i\in C_c^\infty(W_i)\) equal to \(1\) on
a neighborhood of \(p_i\). For \(\alpha\in C^\infty_{\mathrm{odd}}(S^1)\), set
\(\alpha_K=\sum_i S_i(\zeta_i\,\alpha|_{W_i})\). The function
\(\alpha-\alpha_K\) vanishes near \(K\), so we get
\(T(\alpha-\alpha_K)=0\). Hence the restriction of \(\Lambda\) to smooth odd
functions is given by
\[
  \Lambda(\alpha)=T_{\mathrm{odd}}(\alpha)
  =\sum_{i=1}^N T_i(\zeta_i\,\alpha|_{W_i})
  =\sum_{i=1}^N A_i \alpha(p_i)
  \qquad(\alpha\in C^\infty_{\mathrm{odd}}(S^1)).
\]

It remains to pass from smooth odd functions to all of \(\Odd_\sigma(S^1)\).
Set \(M=\Lambda-\sum_i A_i\ev_{p_i}\). This is a smooth linear functional and
vanishes on \(C^\infty_{\mathrm{odd}}(S^1)\). Given
\(h\in\Odd_\sigma(S^1)\), choose the plot \(c\) from
\Cref{thm:smooth-profile-approximation}. Then we have \(c(0)=h\), while
\(c(s)\in C^\infty_{\mathrm{odd}}(S^1)\) for \(s\neq0\). The smooth function
\(M\circ c\) vanishes on the punctured interval, so that \(M(h)=0\). Thus
\(\Lambda\) is the finite linear combination \(\sum_i A_i\ev_{p_i}\) on all
of \(\Odd_\sigma(S^1)\). This proves the smooth-dual formula.

By \Cref{prop:external-right-Ix}, the smooth-dual formula above identifies
the external tangent space with
\(\operatorname{span}\{\ev_v\mid v\in S^1\}\). For each
\(L\in\RP^1\), the rule
\[
  L\longrightarrow
  \operatorname{span}\{\ev_v\mid v\in S^1\cap L\},
  \qquad tv\mapsto t\ev_v
  \quad(v\in S^1\cap L),
\]
is a well-defined isomorphism, since replacing \(v\) by \(-v\) changes both
\(t\) and \(\ev_v\) by a sign. The images for distinct lines are linearly
independent because smooth odd bump functions separate any finite set of
antipodal pairs. Taking their direct sum gives the asserted vector-space isomorphism
\(\external{0}{P}\cong\bigoplus_{L\in\RP^1}L\), completing the proof.
\end{proof}

\begin{cor}\label{cor:tangent-separation}
Under the canonical inclusion of \Cref{prop:external-right-Ix}, the external
tangent space is a proper subspace of the right tangent space:
\[
  \external{0}{P}\subsetneq\righttangent{0}{P}.
\]
\end{cor}

\begin{proof}
By \Cref{prop:external-right-Ix}, the right tangent space is the full
algebraic dual of \(\Odd_\sigma(S^1)\), whereas
\Cref{thm:tangent-computation} identifies the external tangent space with
the finite linear combinations of evaluations. It therefore suffices to
construct an algebraic linear functional that is not such a combination.
For \(n\ge1\), set \(v_n=(\cos(1/n),\sin(1/n))\), and let \(f_n\) be the
odd function supported on \(\{\pm v_n\}\) with \(f_n(v_n)=1\). By the same
argument as in \Cref{ex:discontinuous-sigma-continuous-profile}, each \(f_n\)
belongs to \(\Odd_\sigma(S^1)\). Their supports are pairwise disjoint, so the
family \(\{f_n\}\) is linearly independent. Extend this family to a basis of
the underlying vector space \(\Odd_\sigma(S^1)\), and define
\[
  \Lambda\in\Hom_{\R}\bigl(\Odd_\sigma(S^1),\R\bigr)
\]
by setting \(\Lambda(f_n)=1\) for every \(n\) and \(0\) on the
remaining basis elements. Every finite linear combination of evaluations
vanishes on all but finitely many \(f_n\), whereas
\(\Lambda(f_n)=1\) for every \(n\). Hence \(\Lambda\) is not a finite linear
combination of evaluations, proving the corollary.
\end{proof}

\section*{Acknowledgments}

I would like to express my sincere gratitude to my supervisor, Takuya Sakasai, for his valuable guidance and continuous support throughout this research. I am also grateful to Patrick Iglesias-Zemmour for bringing the pencil diffeology to my attention and for suggesting its connection with my work on non-smooth derivations, and to Toshiyuki Kobayashi for his generous support and encouragement as my supporting supervisor in the WINGS-FMSP program.

This work was supported by JSPS Research Fellowships for Young Scientists and KAKENHI Grant Number JP24KJ0881. I also acknowledge the WINGS-FMSP program for its support.

\bibliographystyle{plain}
\bibliography{pencil-diffeology}

\end{document}